\documentclass{amsart}
\usepackage{amssymb,amscd}

\newcommand{\Q}{{\mathbb Q}}
\newcommand{\Z}{{\mathbb Z}}
\newcommand{\C}{{\mathbb C}}
\newcommand{\R}{{\mathbb R}}
\newcommand{\N}{{\mathbb N}}
\newcommand{\Cl}{\operatorname{Cl}}

\newcommand{\Gal}{\operatorname{Gal}}

\newcommand{\D}{{\mathcal D}}
\newcommand{\PP}{{\mathcal P}}
\newcommand{\OO}{{\mathcal O}}

\newcommand{\fa}{{\mathfrak a}}

\newcommand{\fc}{{\mathfrak c}}
\newcommand{\fd}{{\mathfrak d}}

\newcommand{\fp}{{\mathfrak p}}

\newcommand{\la}{\langle}
\newcommand{\ra}{\rangle}

\begin{document}
\title[Distribution of Irreducibles]{\bf The Distribution
of the Irreducibles in an Algebraic Number Field}

\author{David M. Bradley, Ali E. \"{O}zl\"{u}k, Rebecca A. Rozario,  C. Snyder}

\thanks{This manuscript is an extension of the third author's
Masters thesis at UMaine.  The published version appears in the
Journal of the Australian Mathematical Society, vol.~79 (2007),
no.~3, pp.\ 369--390. URL:
http://www.austms.org.au/Publ/JAustMS/V79P3/c97.html }

\newtheorem{thm}{Theorem}
\newtheorem{lem}{Lemma}
\newtheorem{prop}{Proposition}
\newtheorem{cor}{Corollary}
\newtheorem{df}{Definition}

\begin{abstract} We study the distribution
of  principal ideals generated by  irreducible elements in an
algebraic number field.
\end{abstract}
\maketitle

\section{Introduction}

In an abstract algebra course,  students learn that the concepts
of  prime and  irreducible elements do  not coincide in an
integral domain without unique factorization. Usually, various
examples are given in $\Z[\sqrt{-5}\,]$, for instance, showing the
existence of irreducibles which are not prime. Of course, as every
student knows any prime is irreducible and so generally there are
more irreducibles than primes.

This difference leads naturally to two  questions. First, can one
give a characterization of irreducibles in familiar integral
domains where unique factorization need not hold, such as the ring
of integers in an algebraic number field? Second, how  are the
irreducibles distributed, again in an algebraic number field?

The problem of characterizing irreducibles involves, among many
challenges, a good characterization of all the prime ideals in any
given ideal class of the ideal class group of the field. This has
a particularly nice solution when the Hilbert class field of the
number field is an abelian extension of the field of rational
numbers $\Q$, for class field theory shows us that the solution
involves congruences, modulo certain integers depending on only
the field, for the rational primes contained in the prime ideals.
(As a minor aside, we give a characterization of the irreducibles
and primes in two imaginary quadratic number fields of class
number two in the last section of this paper.) In other cases such
a satisfactory characterization is not known and probably even
nonexistent.

In this note, we study instead the distribution of irreducibles.
First, we give a little background. Let $K$ be an algebraic number
field and denote by $M(x)$ the number of nonassociate irreducible
elements $\alpha$ with $|N_{K/\Q}(\alpha)|\leq x.$ In the 1960's,
J.P. R\'emond, cf. \cite{R}, showed that $$M(x)\sim C\frac{x}{\log
x}(\log\log x)^{D-1},$$ as $x\rightarrow\infty$, where $C$ is a
positive constant not explicitly given and $D$ is the Davenport
constant which is a positive integer depending on only the
structure of the ideal class group of $K$. Now, if we let $P(x)$
denote the number of nonassociate primes $\pi$ with
$|N_{K/\Q}(\pi)|\leq x,$ then by a classical density result
 $$P(x)\sim \frac1{h}\frac{x}{\log x},$$ where $h$ is
the class number of the field, i.e. the order of the ideal class
group. If $h>1$ (so $D>1$, cf. section 2), then there are  ``many
more" irreducibles than primes. If $h=1$, however, then the ring
of integers is a unique factorization domain and hence the
irreducibles and primes coincide. This is consistent with the
estimates above once we observe that in this case, $C=1$ and
$D=1$; see the next section for more on these constants.

Subsequently, J. Kaczorowski, cf. \cite{K}, gave a major extension
of R\'emond's result, which we state here in simplified form:
$$M(x)=\frac{x}{\log x}\left(\sum_{j=0}^{D-1}m_j(\log\log
x)^j\right)+O\left(\frac{x}{\log^2x}(\log\log x)^{c}\right),$$ as
$x\rightarrow\infty$, for some constant $c>0$ and complex numbers
$m_j$. In particular, $m_{D-1}=C$ the coefficient in R\'emond's
estimate. As in R\'emond's case, the constants depend on $K$ but
are not explicitly given.

Later, F. Halter-Koch and W. M\"uller in joint work \cite{H-KM}
showed, among many results, how to determine the constant $C$ and
as a result showed that it depends on only the class group of $K$.

This result prompted us to explore the dependence of some of  the
other coefficients in Kaczorowski's estimate on the arithmetic of
$K$. In particular, we consider $m_{D-2}$ and give an explicit
expression for this coefficient. We then apply this to the special
case of a number field with cyclic class group in which case we
find that $m_{D-2}$ contains explicit arithmetic information about
the field and some of the subfields of its Hilbert class field.
(We chose the case of cyclic class group due to the messy
combinatorical arguments in the general case. It would still
perhaps be  of interest to see what happens in general.) Finally,
we compute---more precisely, approximate---$m_{D-2}$ for two
imaginary quadratic number fields with class number two. Indeed,
this calculation shows that more than just properties of the class
group figure into the makeup of $m_{D-2}.$

\section{A Dirichlet Series Associated with Irreducibles}

Let $K$ be an algebraic number field, i.e. a finite extension of
the rational number field, $\Q$, and let $\OO_K$ denote its ring
of integers. We denote by $N(x)$ the norm of an element $x$ from
$K$ to $\Q$.  Also, we denote by $N\fa$ the norm of an ideal $\fa$
of $\OO_K$. Furthermore, let $\Cl=\Cl(K)$ denote the class group
of $K$ and $h=h_K$ the class number, i.e. the order of $\Cl(K)$.

In studying the distribution of the irreducibles, we introduce the
following function.

\begin{df} $$\mu(s)=\sum_{(\alpha)\atop \alpha~
irred.}|N(\alpha)|^{-s},$$ where $s$ is a complex number with real
part, $\sigma >1$. \end{df}
 The sum runs over the principal ideals generated by
 irreducible elements of $\OO_K$. We obviously do not wish to
 count all associates of an irreducible since there are
 infinitely many when the unit group is infinite, i.e.
anytime $K$ is not $\Q$ or an imaginary quadratic number field.

Ultimately, we shall be interested in the ``summatory" function
given by
\begin{df} $$M(x)=\sum_{{(\alpha) \atop \alpha~ irred.}\atop
|N(\alpha)|\leq x}1, $$ where $x$ is any positive real number.
\end{df}

We shall first determine  properties of $\mu(s)$  and then use a
well-known Tauberian theorem to glean information about the
distribution of $M(x)$.

To this end, consider the following.  Write
$\Cl=\{\fc_1=1,\fc_2,\cdots,\fc_h\}$.

\begin{df} For each positive integer $m$, let
$$\mathcal{D}_m=\{\underline{k}=(k_1,\cdots,k_h)\in \N_0^h~:
~\prod_{j=1}^h\fc_j^{k_i}\overset{min}{=}1,~k_1+\cdots+k_h=m\},$$
where $\prod \fc_i^{k_i}\overset{min}{=}1$ means that $\prod
\fc_i^{k_i}=1$ and if $\prod \fc_i^{\ell_i}=1$ for some $\ell_i$
such that $0\leq \ell_i \leq k_i$ for $i=1,\cdots,h$, then
$\ell_i=0$ for all $i$ or $\ell_i=k_i$ for all $i$. (Here $\N_0$
denotes the set of nonnegative integers.)
\end{df}

Notice that $\overset{min}{=}$ guarantees that a product of
elements is $1$ but no nontrivial subproduct is $1$. Hence the
product gives a ``minimal" representation of $1$.

Later on it will be more convenient to think of the elements of
$\mathcal{D}_m$ as functions in the usual way; namely,
$$\mathcal{D}_m=\{\kappa : \Cl \longrightarrow \N_0\; |\;
\prod_{\fc\in\Cl}\fc^{\kappa(\fc)}\overset{min}{=}1,\, \sum_\fc
\kappa(\fc)=m\}.$$

\begin{df} The Davenport constant of  $~\Cl$, denoted by $D$ or
$D(\Cl)$, is the largest positive integer $m$ such that $\D_m$ is
nonempty.
\end{df}

The Davenport constant is defined as above for any finite abelian
group. In general, the relation  between the Davenport constant
and the structure of the group is not known. On the other hand, it
is well known (and easy to prove) that the Davenport constant is
no larger than the order of the group.

 We now have the following
proposition which gives a connection between irreducibles and
prime ideals. First, we denote the set of nonzero prime ideals of
$\OO_K$ by $\PP$.

\begin{prop}\label{P1}
$$\mu(s)=\sum_{m=1}^D~\sum_{\underline{k}\in\D_m}~ \prod_{i=1}^h
\sum_{{\fa_i \atop \exists \, \fp_{i1},\cdots,\fp_{ik_i}\in
\PP\cap\fc_i}\atop \fa_i=\fp_{i1}\cdots\fp_{ik_i}}N(\fa_i)^{-s},$$
for any complex $s$ with $\sigma>1$ and where $\sum_{\fa_i}$ is
defined to be $1$ whenever $k_i=0$.
\end{prop}
\begin{proof}  For $\underline{k}\in\mathcal{D}_m$, define
$$\mathcal{A}_{\underline{k}}=\{\fa~:~\fa=\fa_1\cdots\fa_h,
\fa_i=\fp_{i1}\cdots\fp_{ik_i},~~ \mbox{some}~~
\fp_{ij}\in\mathcal{P}\cap \fc_i\},$$ where $\fa_i=1$, if $k_i=0$.
Now let $\mathcal{A}=\cup \mathcal{A}_{\underline{k}}$ where the
union is over all $\underline{k}$ in $\cup_m \mathcal{D}_m$. By
the uniqueness of the factorization of ideals into prime ideals,
we see that this union is disjoint. Moreover, by the
multiplicativity of the norms, we have
$$\sum_{m=1}^D\sum_{\underline{k}\in\mathcal{D}_m}\prod_{i=1}^h
\sum_{\fa_i}N\fa_i^{-s}=\sum_{\fa\in\mathcal{A}}N\fa^{-s},$$ where
 $\fa_i$ are as above in the definition of
$\mathcal{A}_{\underline{k}}$. Now notice that if
$\fa\in\mathcal{A}$, then $\fa\in\mathcal{A}_{\underline{k}}$ for
some $\underline{k}\in\mathcal{D}_m$. Thus the ideal class $[\fa]$
containing $\fa$ satisfies
$$[\fa]=\prod_{i=1}^h{\fc_i}^{k_i}\overset{min}{=}1,$$ by definition
of $\mathcal{D}_m$. Hence $\fa=(\alpha)$ for some nonzero, nonunit
integer $\alpha$ in $K$. But notice that $\alpha$ must be
irreducible for otherwise $$[\fa]=\prod_{i=1}^h{\fc_i}^{k_i}=1,$$
would not be a minimal representation of $1$.

Conversely, if $\alpha$ is irreducible, then
$(\alpha)\in\mathcal{A}_{\underline{k}}$ for some $\underline{k}$;
namely, $$(\alpha)=\prod_{i=1}^h\prod_{j=1}^{k_i}\fp_{ij},$$ for
some $k_i\in\N_0$ and $\fp_{ij}\in\mathcal{P}\cap\fc_i.$
\end{proof}

Next, we examine the right-hand sum in the proposition above. To
this end we define the following family of polynomials.

\begin{df} Let $k$ be a positive integer and $z_1,\cdots,z_k$
independent variables. Then
$$P_k(\underline{z})=P_k(z_1,\cdots,z_k)=
\sum_{(\nu_1,\cdots,\nu_k)\in\N_0^k\atop \sum
j\nu_j=k}\frac1{\nu_1!\cdots\nu_k!\, 1^{\nu_1}\cdots
k^{\nu_k}}z_1^{\nu_1}\cdots z_k^{\nu_k}.$$ Moreover, let
$$P_0(\underline{z})=1.$$
\end{df}

\begin{prop}\label{P2} Let $k$ be a positive integer
and $x_1,x_2,x_3,\cdots$ be a sequence of independent variables.
Moreover, for $j=1,\cdots, k$, let $s_j=\sum_{i=1}^\infty x_i^j.$
Then
$$\sum_{(n_1,\cdots,n_k)\in\N^k \atop n_1\leq\cdots\leq n_k}x_{n_1}\cdots
x_{n_k}=P_k(s_1,\cdots,s_k).$$
\end{prop}

\begin{proof} First we introduce some notation.  Let
$\underline{x}_{\underline{n}}= x_{n_1}\cdots x_{n_k}$\\ for any
$\underline{n}=(n_1,\cdots,n_k)\in\N^k.$ Let
$T=\{\underline{n}\in\N^k:n_1\leq\cdots\leq n_k\}.$ Also, let
$S_k$ be the symmetric group on $\{1,\cdots,k\}$; for $\sigma\in
S_k$, let
$\sigma\underline{n}=(n_{\sigma(n_1)},\cdots,n_{\sigma(n_k)}).$
Next, let $C=C(\sigma)$ be the conjugacy class of $\sigma$ in
$S_k$, i.e. $C(\sigma)=\{\gamma\sigma\gamma^{-1}: \gamma\in
S^k\}.$ Let $$\sigma=\prod_{j=1}^k\eta_{j1}\cdots\eta_{j\nu_j}$$
be a factorization of $\sigma$ into disjoint cycles, where
$\nu_j\in \N_0$ and for each $j$ and  $i=1,\cdots,\nu_j$, the
permutations $\eta_{ji}$ are the distinct $j$-cycles, say
$\eta_{ji}=(a_{ji1}\cdots a_{jij})$ with $a_{ji\ell}\in
\{1,\cdots,k\},$ and with the convention that $1$-cycles are
included so that $\cup_{j,i}\{a_{ji1},\cdots,
a_{jij}\}=\{1,\cdots,k\}.$ Recall that $\tau\in C(\sigma)$ if and
only if $\tau$ has the same type of cycle decomposition, i.e. if
$$\tau=\prod_{j=1}^k\eta_{j1}'\cdots\eta_{j\nu'_j}'$$ into
disjoint cycles with the same conventions as above, then
$\nu'_j=\nu_j$ for $j=1,\cdots,k,$ (see, for example \cite{DF}).
Notice then that a conjugacy class in $S_k$ is determined uniquely
by a $k$-tuple, $(\nu_1,\cdots, \nu_k)\in {\N_0^k}$ with
$\sum_{j=1}^kj\nu_j=k.$ Any permutation in the conjugacy class has
a cycle decomposition determined by the $\nu_j$'s as above.
Moreover, recall that $$\# C(\sigma) =\frac{k!}{\nu_1!\cdots
\nu_k! 1^{\nu_1}\cdots k^{\nu_k}},$$ again see \cite{DF}.
Furthermore, recall that the cardinality of the orbit of
$\underline{n}$ under $S_k$,
$S_k\underline{n}=\{\eta\underline{n}:\eta\in S_k\},$ is equal to
$|S_k|/|S_k(\underline{n})|$ where $S_k(\underline{n})=\{\eta\in
S_k: \eta\underline{n}=\underline{n}\}$, the stabilizer  subgroup
of $\underline{n}.$  Moreover, if $\underline{m}\in
S_k\underline{n}$, then the stabilizer subgroups,
$S_k(\underline{m})$ and $S_k(\underline{n})$, are conjugate and
thus have the same cardinality.

Now for the proof: Notice that
$$\frac1{k!}\sum_{\sigma\in S_k}\sum_{\underline{m}\in\N^k\atop
\sigma\underline{m}=\underline{m}}\underline{x}_{\underline{m}}=
\frac1{k!}\sum_{C}\sum_{\sigma\in
C}\sum_{\underline{m}\in\N^k\atop
\sigma\underline{m}=\underline{m}}\underline{x}_{\underline{m}},$$
where $\sum_C$ is the sum over the conjugacy classes of $S_k$. Now
notice that if we write
$\sigma=\prod_{j=1}^k\eta_{j1}\cdots\eta_{j\nu_j}$ as above, then
$$\sum_{\underline{m}\in\N^k\atop
\sigma\underline{m}=\underline{m}}\underline{x}_{\underline{m}}=s_1^{\nu_1}
\cdots s_k^{\nu_k},$$ which is independent of the choice of
$\sigma\in C.$ Hence
$$\frac1{k!}\sum_{C}\sum_{\sigma\in
C}\sum_{\underline{m}\in\N^k\atop
\sigma\underline{m}=\underline{m}}\underline{x}_{\underline{m}}=
\frac1{k!}\sum_{C}|C|\sum_{\underline{m}\in\N^k\atop
\sigma\underline{m}=\underline{m}}\underline{x}_{\underline{m}}$$
$$ =\frac1{k!} \sum_{(\nu_1,\cdots,\nu_k)\in\N_0^k\atop
\sum_j\nu_j=k}\frac{k!}{\nu_1!\cdots \nu_k! 1^{\nu_1}\cdots
k^{\nu_k}}s_1^{\nu_1}\cdots s_k^{\nu_k}=P_k(s_1,\cdot,s_k).$$

On the other hand, $$\sum_{\underline{n}\in
T}\underline{x}_{\underline{n}}=\sum_{\underline{n}\in
T}\frac1{|S_k\underline{n}|} \sum_{\underline{m}\in
S_k\underline{n}}\underline{x}_{\underline{m}},$$ since
$\underline{x}_{\underline{m}}=\underline{x}_{\underline{n}}$ for
any $\underline{m}\in S_k\underline{n}.$ Hence
$$\sum_{\underline{n}\in T}\underline{x}_{\underline{n}}=
\sum_{\underline{m}\in\N^k}\sum_{\underline{n}\in T\atop
\underline{m}\in
S_k\underline{n}}\frac1{|S_k\underline{n}|}\;\underline{x}_{\underline{m}}=
\frac1{k!}\sum_{\underline{m}\in\N^k}\sum_{\underline{n}\in T\atop
\underline{m}\in
S_k\underline{n}}|S_k(\underline{n})|\;\underline{x}_{\underline{m}}$$
$$=\frac1{k!}\sum_{\underline{m}\in\N^k}\sum_{\underline{n}\in
T\atop \underline{m}\in
S_k\underline{n}}|S_k(\underline{m})|\underline{x}_{\underline{m}}=
\frac1{k!}\sum_{\underline{m}\in\N^k}\sum_{\underline{n}\in T\atop
\underline{m}\in S_k\underline{n}}\sum_{\sigma\in
S_k(\underline{m})}\underline{x}_{\underline{m}}$$ $$=
\frac1{k!}\sum_{\underline{m}\in\N^k}\sum_{\sigma\in
S_k(\underline{m})}\underline{x}_{\underline{m}}\sum_{\underline{n}\in
T\atop \underline{m}\in S_k\underline{n}}1.$$ But
$$\sum_{\underline{n}\in T\atop \underline{m}\in
S_k\underline{n}}1=1,$$ since only one permutation of
$\underline{m}$ can belong to $T$. Therefore,
$$\sum_{\underline{n}\in
T}\underline{x}_{\underline{n}}=\frac1{k!}\sum_{\underline{m}\in\N^k}
\sum_{\sigma\in S_k\atop
\sigma\underline{m}=\underline{m}}\underline{x}_{\underline{m}}=P_k(s_1,
\cdots, s_k),$$ from above, as desired.

\end{proof}

\begin{prop}\label{P3} Let $k$ be a nonnegative integer
and $\fc$ any class in $\Cl$. Then $$\sum_{{\fa \atop \exists \,
\fp_{1},\cdots,\fp_{k}\in \PP\cap\fc}\atop
\fa=\fp_{1}\cdots\fp_{k}}N(\fa)^{-s}= P_k(\underline{z}),$$ where
$$z_j=\sum_{\fp\in\PP\cap\fc}N\fp^{-js},$$ for all
$Re(s)=\sigma>1.$
\end{prop}
\begin{proof} For $k=0$, both sides are equal to 1, for the
left-hand side consists of one term, $\fa=\OO_K$ which has norm
equal to 1.

Assume $k>0.$ Write $\mathcal{P}\cap \fc=\{\fp_n:n\in\N\}.$ For
any $n\in\N$, let $x_n=N\fp_n^{-s}$. Then the proposition follows
directly from Proposition \ref{P2}, once we observe that $N$ is
multiplicative and all series involved converge absolutely, since
$\sigma>1$.
\end{proof}

 We now have the following useful corollary to
Proposition \ref{P3}.

\begin{cor}\label{C1} $$\mu(s)=\sum_{(\alpha)\atop \alpha~
irred.}|N(\alpha)|^{-s}=\sum_{m=1}^D~\sum_{\underline{k}\in\D_m}~
\prod_{i=1}^hP_{k_i}(z_{i1},\cdots,z_{ik_i}),$$ where
$$z_{ij}=\sum_{\fp_i\in\PP\cap\fc_i}N\fp_i^{-js}.$$
\end{cor}

For the next proposition, write
$$z_{i1}=\sum_{\fp_i\in\PP\cap\fc_i}N\fp_i^{-s}=\ell +g_i,$$ where
$$\ell=\frac1{h}\log(\frac1{s-1}),$$ and $$g_i=g_i(s).$$ It is
well known that $g_i(s)$ is regular at $s=1.$  We then have

\begin{prop} \label{P4} $$\mu(s)=\sum_{\mu=0}^Dc_\mu~ \ell^\mu,$$
 where $$c_\mu=\sum_{m=\max(1,\mu)}^D~\sum_{\underline{k}\in\D_m}
a_{\underline{k},\mu},$$ where if
$\underline{k}=(k_1,\cdots,k_h)$, then
$$a_{\underline{k},\mu}=\underset{\mu_1+\cdots+\mu_h=\mu}
{\sum_{\mu_1=0}^{k_1}\cdots\sum_{\mu_h=0}^{k_h}}~
\prod_{i=1}^hb_{{k_i},\mu_i},$$ with
$$b_{{k_i},\mu_i}=\sum_{\nu_{i1}=\mu_i}^{k_i}
\frac{g_i^{\nu_{i1}-\mu_i}}{\mu_i!\,(\nu_{i1}-\mu_i)!}~
\rho_{k_i,\nu_{i1}},$$  where
$$\rho_{k_i,\nu_{i1}}=
\sum_{(\nu_{i2},\cdots,\nu_{ik_i})\in\N_0^{k_i-1}\atop \sum
j\nu_{ij}=k_i-\nu_{i1}}\frac{1}{\nu_{i2}!\cdots \nu_{ik_i}!~
2^{\nu_{i2}}\cdots k_i^{\nu_{ik_i}}}z_{i2}^{\nu_{i2}}\cdots
z_{ik_i}^{\nu_{ik_i}},$$ if $k_i>1$,  and we define
$\rho_{0,0}=1$, $\rho_{1,1}=1,$ and $\rho_{1,0}=0$.
\end{prop}
\begin{proof}  First use the definition of the polynomials
$P_k(\underline{z})$ to expand $\mu(s)$ in Proposition \ref{P3},
where the indices of summation are $\nu_{ij}$ for $i=1,\cdots,h$
and $j=1,\cdots,k_i$. Hence $$\mu(s)=$$
$$\sum_{m=1}^D\sum_{(k_1,\cdots,k_h)\in\mathcal{D}_m}
\prod_{i=1}^h\sum_{(\nu_{i1},\cdots,\nu_{ik_i})\atop \sum
j\nu_{ij}=k_i}\frac1{\nu_{i1}!\cdots\nu_{ik_i}!1^{\nu_{i1}}\cdots
k!^{\nu_{ik_i}}}(\ell+g_i)^{\nu_{i1}}{z_{2i}}^{\nu_{i2}}\cdots
{z_{k_ii}}^{\nu_{ik_i}},$$  where
$\sum_{(\nu_{i1},\cdots,\nu_{ik_i})}\cdots =1$, if $k_i=0$. Now in
the right-hand most sum above, sum  over the $\nu_{i1}$ first in
which case we get $$\sum_{(\nu_{i1},\cdots,\nu_{ik_i})\atop \sum
j\nu_{ij}=k_i}\frac1{\nu_{i1}!\cdots\nu_{ik_i}!1^{\nu_{i1}}\cdots
k!^{\nu_{ik_i}}}(\ell+g_i)^{\nu_{i1}}{z_{2i}}^{\nu_{i2}}\cdots
{z_{k_ii}}^{\nu_{ik_i}}=$$  $$
\sum_{\nu_{i1}=0}^{k_i}\frac{(\ell+g_i)^{\nu_{i1}}}{\nu_{i1}!}
\rho_{k_i,\nu_{i1}},$$ with $\rho$ as defined in the statement of
the proposition. Next, expand
$z_{i1}^{\nu_{i1}}=(\ell+g_i)^{\nu_{i1}}$ as
$$\sum_{\mu_i=0}^{\nu_{i1}}{\nu_{i1}\choose
\mu_i}\ell^{\mu_i}g_i^{\nu_{i1}-\mu_i}.$$ Then
$$\sum_{\nu_{i1}=0}^{k_i}\frac{(\ell+g_i)^{\nu_{i1}}}{\nu_{i1}!}
\rho_{k_i,\nu_{i1}}=\sum_{\nu_{i1}=0}^{k_i}\frac1{\nu_{i1}!}
\sum_{\mu_i=0}^{\nu_{i1}}\left(\nu_{i1}\atop
\mu_i\right)g_i^{\nu_{i1}-\mu_i}\rho_{k_i,\nu_{i1}}\ell^{\mu_i}=
\sum_{\mu_i=0}^{k_i}b_{k_i,\mu_i}\ell^{\mu_i},$$ where the $b$ are
defined as above.  But then
$$\prod_{i=1}^h\sum_{\mu_i=0}^{k_i}b_{k_i,\mu_i}\ell^{\mu_i}=
\sum_{\mu_1=0}^{k_1}\cdots\sum_{\mu_h=0}^{k_h}
\prod_{i=1}^hb_{k_i,\mu_i}\ell^{\mu_1+\cdots+\mu_h}=
\sum_{\mu=0}^ma_{\underline{k},\mu}\ell^\mu,$$ with the $a$ as
defined above.

But now $$\sum_{\underline{k}\in
\mathcal{D}_m}\sum_{\mu=0}^ma_{\underline{k},\mu}\ell^\mu=
\sum_{\mu=0}^m\sum_{\underline{k}\in
\mathcal{D}_m}a_{\underline{k},\mu}\ell^\mu.$$ Hence
$$\mu(s)=\sum_{m=1}^D\sum_{\mu=0}^m\sum_{\underline{k}\in
\mathcal{D}_m}a_{\underline{k},\mu}\ell^\mu=\sum_{\mu=0}^D
\left(\sum_{m=\max (1,\mu)}^D\sum_{\underline{k}\in
\mathcal{D}_m}a_{\underline{k},\mu}\right)\ell^\mu,$$ as desired.
\end{proof}

Now we rewrite the $a_{\underline{k},\mu}$ in Proposition \ref{P4}
in a form more convenient for winning an explicit formula for
$c_\mu$ for ``large" $\mu$.

\begin{cor}\label{C2} $$\mu(s)=\sum_{\mu=0}^Dc_\mu~ \ell^\mu,$$
where
$$c_\mu=\sum_{\nu=max(1,\mu)-\mu}^{D-\mu}\;\sum_{\underline{k}\in\D_{\mu+\nu}}
a_{\underline{k},\mu},$$ with
$$a_{\underline{k},\mu}=\underset{\nu_1+\cdots+\nu_h=\nu}
{\sum_{\nu_1=0}^{k_1}\cdots\sum_{\nu_h=0}^{k_h}}~
\prod_{i=1}^h\frac1{k_i!}~ \sum_{\lambda_i=0}^{\nu_i}\frac{k_i!}
{(\nu_i-\lambda_i)!(k_i-\nu_i)!}g_i^{\nu_i-\lambda_i}
\rho_{k_i,k_i-\lambda_i},$$ where (as above)
$$\rho_{k_i,k_i-\lambda_i}=
\sum_{(\nu_{i2},\cdots,\nu_{ik_i})\in\N_0^{k_i-1}\atop \sum
j\nu_{ij}=\lambda_i}\frac{1}{\nu_{i2}!\cdots \nu_{ik_i}!~
2^{\nu_{i2}}\cdots k_i^{\nu_{ik_i}}}z_{i2}^{\nu_{i2}}\cdots
z_{ik_i}^{\nu_{ik_i}}.$$
\end{cor}
\begin{proof} (Sketch) In Propostion \ref{P4} change variables as
follows: let $\nu=m-\mu,$ let $\nu_i=k_i-\mu_i$, and let
$\lambda_i=k_i-\nu_{i1}.$
\end{proof}

 From this corollary we extract the following result.
\begin{cor}\label{C3} Let $$\mu(s)=\sum_{\mu=0}^Dc_\mu\,\ell^\mu.$$ Then

i) $$c_D=\sum_{\underline{k}\in\D_D}\,\prod_{i=1}^h\frac1{k_i!}.$$

ii)
$$c_{D-1}=\sum_{\underline{k}\in\D_{D-1}}\prod_{i=1}^h\frac1{k_i!}~
+~\sum_{\underline{k}\in\D_D}\,
\prod_{i=1}^h\frac1{k_i!}\,\sum_{j=1}^hk_jg_j.$$

iii) If $D\geq 2$, then
$$c_{D-2}=\sum_{\underline{k}\in\D_{D-2}}\,
\prod_{i=1}^h\frac1{k_i!}~+~\sum_{\underline{k}\in\D_{D-1}}\,
\prod_{i=1}^h\frac1{k_i!}\sum_{j=1}^hk_jg_j$$
$$+~\sum_{\underline{k}\in\D_D}\,\prod_{i=1}^h\frac1{k_i!}
\left(\sum_{1\leq j_1< j_2\leq
h}k_{j_1}k_{j_2}g_{j_1}g_{j_2}\,+\,\sum_{j=1}^hk_j(k_j-1)
\left(\frac1{2}g_j^2\,+\,\frac1{2}z_{j2}\right)\right).$$
\end{cor}

The proof is a straightforward application of the previous
corollary.

We further obtain the following expressions for $\mu(s)$ for some
fields with small class number.

\begin{cor}\label{C4}

i) Suppose $D=1$ whence $h=1$. Then  $$\mu(s)=\ell+g_1.$$

ii) If $D=2$ so $h=2$, say $\Cl=\{1=\fc_1,\fc_2\},$ then
$$\mu(s)=\frac1{2}\ell^2+(1+g_2)\ell +\left(g_1+\frac1{2}g_2^2+
\frac1{2}z_{22}\right).$$
\end{cor}
\begin{proof} In light of the formulas for the $c_\mu$ above, it
suffices to compute $\D_m$ for each of the groups listed.

Let $\Cl=\{1=\fc_1 \}$. Then we have only one minimal
representation of 1, namely $1\overset{min}{=}1$, implying that
$\D_1=\{1\}$. Using this with the previous corollary yields {\it
i)}.

Now let $\Cl=\{1=\fc_1,a=\fc_2\}$. Then we have two minimal
representations of 1, namely, $1 \overset{min}{=}1$, and
$aa\overset{min}{=}1$ implying that $\D_1=\{(1,0)\}$ and
$\D_2=\{(0,2)\},$ respectively. This yields {\it ii)}.
\end{proof}

\section{The Summatory Function $M(x)$}

Having established formal properties of the Dirichlet series
$\mu(s)$, we now use well-known results relating a Dirichlet
series to its associated summatory function as in \cite{K}. We
present the following weaker form of Kaczorowski's ``Main Lemma"
given in \cite{K}, which will be sufficiently strong for our
purposes.

Let $$f(s)=\sum_{n=1}^\infty\frac{a_n}{n^s}$$ be a Dirichlet
series where $s=\sigma+it$ with $a_n,\sigma,t$ real numbers and
$a_n\geq 0$.

As in \cite{K} we have the following definition.

\begin{df} We let $\mathcal{A}$ be the set of those Dirichlet series $f$
as above satisfying the following three additional properties:

 (i) for all $x,y\in \R$ such that $1\leq x <y$,

$$\sum_{x\leq n\leq y}a_n\leq (y-x)\log^{c_1}y +O(y^\theta),$$ for
some $c_1>0,~~ \theta <1$ where the constants depend on $f$ only.

 (ii) There exists a nonnegative  integer $k$ and
functions $g_j(s)$ for $j=0,\cdots,k$, such that
$$f(s)=\sum_{j=0}^kg_j(s)\log^j\left(\frac1{s-1}\right),$$ for
$\sigma > 1$ and such that $g_k(1)\not =0$ and $g_j(s)$ is regular
for $\sigma
>1$ and can be analytically continued to a regular function in
 the region
$\mathcal{R}$ given by $$\mathcal{R}=\left\{s=\sigma
+it~:~\sigma>1-\frac{c_2}{\log(|t|+2)}\right\}$$ for some $c_2>0$.

(iii) In the region $\mathcal{R}$ $$|g_j(s)|\ll
\log^{c_3}(|t|+3),$$ for some $c_3>0.$ \end{df}

\begin{prop}\label{ML} {\rm (Corollary to Kaczorowski's Main Lemma)}
Let\\ $f(s)=\sum_{n=1}^\infty a_nn^{-s}$ be a Dirichlet series in
class $\mathcal{A}$ as defined above. Let $S(x)=\sum_{n\leq
x}a_n,$ the summatory function associated with $f(s)$. Then for
all $\epsilon>0$ and all $x\geq e^e$,
$$S(x)=\frac{x}{\log x}\left(\sum_{j=0}^{k-1}e_j(\log\log
x)^j\right)+ O\left(\frac{x}{\log^{2-\epsilon}x}\right),$$ as
$x\rightarrow\infty$, where the $e_j$ are complex numbers given by
$$e_j=\sum_{\nu=j}^k\frac{\nu!}{j!}\,g_\nu(1)\,I_{\nu -j},$$ with
$$I_m= \frac{(-1)^m}{m!}\frac1{2\pi
i}\int_\mathcal{C}e^z(\log z)^m~dz,$$ where $\mathcal{C}$ is the
path of integration consisting of the segment $(-\infty,-1]$ on
the lower side of the real axis (so that the argument of $\log z$
is $-\pi$), the circumference of the unit circle taken
counter-clockwise, and the segment $[-1,-\infty)$ on the upper
side of the real axis.
\end{prop}

The proof may be found in \cite{K} where we take Case I and $q=0$
in the Main Lemma.

\begin{lem}\label{L3} Let $t$ be any positive real number with $t<1$. Then

a) $I_0=0$,

b) $\sum_{m=1}^\infty t^{m-1}\,I_m=\exp\big(\gamma
t+\sum_{n=2}^\infty (-1)^{n-1}\zeta(n)\frac{t^n}{n}\big)$, where
$\gamma=0.577\ldots$ is Euler's constant,

c) $I_1=1$ and $I_2=\gamma$.

\end{lem}
\begin{proof} Part a) follows since
$I_0=\int_\mathcal{C}e^z~dz=0.$

With respect to Part b), consider the formal sum
$$\sum_{m=0}^\infty t^m\,I_m=\frac{1}{2\pi
i}\int_\mathcal{C}e^ze^{-t\log z}\,dz=\frac1{2\pi
i}\int_\mathcal{C}e^ze^{-t}\,dz=\frac1{\Gamma(t)}.$$ But then
since $I_0=0$, we have $$\sum_{m=1}^\infty
t^{m-1}\,I_m=\frac1{t\Gamma(t)}=\exp\big(\gamma
t+\sum_{n=2}^\infty (-1)^{n-1}\zeta(n)\frac{t^n}{n}\big),$$ by
\cite{Wr}.

Part c) follows immediately from b).
\end{proof}

\begin{cor}\label{C5} Let $e_j$ be defined as in the proposition above. Then

i) if $k\geq 1$, $$e_{k-1}=k\,g_k(1),$$

ii) if $k\geq 2$, $$e_{k-2}= (k-1)g_{k-1}(1)
+k(k-1)g_k(1)\,\gamma.$$

\end{cor}

The proof is immediate from the preceding lemma and proposition.

We now apply these results to $\mu(s)$ to obtain information about
$M(x)$.  By \cite{K}, using results in \cite{L},  $\mu(s)$ belongs
to the class $\mathcal{A}$.

We shall state a well-known result about $\sum_{\fp\in
\fc}\frac1{N\fp^s},$ for $\fc\in \Cl$, but first we recall some
definitions.

Let $K$ be an algebraic number field of degree $n$ over $\Q$ with
class group $\Cl(K)=\Cl$ of order $h$. Let $\widehat{\Cl}$ denote
the character group of $\Cl$, i.e. the group of homomorphisms from
$\Cl$ into  the multiplicative group $\C^*$. As usual, we denote
the principal character, i.e. the constant character $1$, by
either $\chi_0$ or simply by $1$.

Let $\chi$ be an arbitrary character on $\Cl$, then we define the
$L$-series
$$L(s,\chi)=\sum_{\fa}\frac{\chi(\fa)}{N\fa^s}\;\;\;\;(\sigma
>1) ,$$ where the sum is over all (nonzero) integral ideals of $K$.

If $\chi=1$, the principal character, then
$$L(s,\chi_0)=\zeta_K(s),$$ the Dedekind zeta function of $K$.

As is well known, $L(s,\chi)$ converges absolutely and uniformally
on compact subsets in the half plane $\sigma>1$. Moreover, since
the norm map $N$ is completely multiplicative on the set of ideals
of $K$, we have
$$L(s,\chi)=\prod_{\fp}\left(1-\frac{\chi(\fp)}{N\fp^s}\right)^{-1},$$ for
all $\sigma>1$ and where the product is taken over all (nonzero)
prime ideals of $K$. It is also well known that in the half plane
$\sigma>1-1/n$, the series for $L(s,\chi)$ converges, if
$\chi\not=1$, and $L(s,\chi)$ is regular there. On the other hand,
$\zeta_K(s)$  has a continuation into the same half plane but with
a simple pole at $s=1$ with (nonzero) residue $a_K$.

Furthermore, in the region $\mathcal{R}_K$ given by
$$\sigma>1-\frac{c_K}{\log(|t|+2)}$$ $L(s,\chi)$ does not vanish,
where $c_K$ depends on $K$ but not on $\chi$.

Now, since $L(s,\chi)$ is nonzero in the region above, we see that
$\log L(s,\chi)$ is defined and regular in this region.

\begin{prop}\label{P5} Let $\fc$ be an ideal class of $\Cl$. Then
$$\sum_{\fp\in\fc}\frac1{N\fp^s}=\frac1{h}\log
\zeta_K(s)+\frac1{h}\sum_{\chi\atop
\chi\not=1}\overline{\chi}(\fc) \log
L(s,\chi)-\sum_{m=2}^\infty\sum_{\fp\atop \fp^m\in
\fc}\frac1{mN\fp^{ms}},$$ for $\sigma>1$.
\end{prop}

For a proof see, for example \cite{Lg}, (or just about any text on
algebraic number theory).

Notice that this proposition allows us to analytically continue
$\sum_{\fp\in\fc}N\fp^{-s}$ onto the region $\mathcal{R}_K$.

\begin{cor}\label{C5.1} Let
$$g_\fc(s)=\sum_{\fp\in\fc}\frac1{N\fp^s}-\frac1{h}
\log\left(\frac1{s-1}\right).$$ Then $$g_\fc(s)=\frac1{h}\log
((s-1)\zeta_K(s))+\frac1{h}\sum_{\chi\atop
\chi\not=1}\overline{\chi}(\fc) \log
L(s,\chi)-\sum_{m=2}^\infty\sum_{\fp\atop \fp^m\in
\fc}\frac1{mN\fp^{ms}},$$ hence regular in $\mathcal{R}_K$. In
particular, $$g_\fc(1)=\frac1{h}\log a_K +\frac1{h}\sum_{\chi\atop
\chi\not=1}\overline{\chi}(\fc) \log
L(1,\chi)-\sum_{m=2}^\infty\sum_{\fp\atop \fp^m\in
\fc}\frac1{mN\fp^{m}},$$ where $a_K$ is the residue of
$\zeta_K(s)$ at $s=1$.
\end{cor}
\begin{proof} Write $\zeta_K(s)$ as $\frac1{s-1}(s-1)\zeta_K(s)$
and then apply $\log$.
\end{proof}

We now apply this result to $M(x)$.

\begin{prop}\label{P4.5} Let $K$ be an algebraic number field with class number
$h$ and associated Davenport number $D$. Then
$$M(x)=Dc_Dh^{-D}\frac{x}{\log x}\left(\log\log x\right)^{D-1}+
\frac{x}{\log x}\sum_{j=0}^{D-2}e_j\left(\log\log x\right)^j $$
$$+O\left(\frac{x}{(\log x)^{3/2}}\right),$$ where the $e_j$ are
given in {\rm Proposition \ref{ML}} with $g_j(s)=h^{-j}c_j(s)$.
\end{prop}
\begin{proof} The proof is immediate since
$$\mu(s)=\sum_{\mu=0}^Dc_\mu(s)\left(\frac1{h}
\log(\frac1{s-1})\right)^\mu.$$
\end{proof}

As an immediate corollary we have,

\begin{cor}\label{C6}
$$M(x)\sim Dc_Dh^{-D}\frac{x}{\log x}\left(\log\log
x\right)^{D-1}.$$
\end{cor}

Compare this with Theorem 1 of \cite{H-KM}.

But we also get the following result.

\begin{thm}\label{T1} For $D\geq 2$,
$$M(x)=\frac{x}{\log x}\left(C(\log\log x)^{D-1}+B(\log\log
x)^{D-2}\right)$$ $$+O\left(E(x)\right),$$ where

$$C=D\,c_D\,h^{-D}$$
$$B=(D-1)c_{D-1}(1)h^{1-D}+D(D-1)c_D h^{-D}\gamma,$$ with
$\gamma$, Euler's constant, and where
$$E(x)=\frac{x}{\log x}(\log\log x)^{D-3}$$ if $D\geq 3,$

and
$$\frac{x}{(\log x)^{3/2}}$$ if not.
\end{thm}

\section{The Special Case of Number Fields with Cyclic Class Group}

We now investigate the asymptotic behavior of $M(x)$ when the
number field $K$ has cyclic class group $\Cl$ of order $h>1$. Then
we see by Theorem \ref{T1} that in order to compute the
coefficients $C$ and $B$, we need to determine $c_D$ and
$c_{D-1}(s).$ First of all, notice that $D=h$, for we have already
observed that $D\leq h$ for any $\Cl$. But now since $\Cl$ is
cyclic generated by $\fc$, say, then $\fc^h\overset{min}{=}1$,
whence $h\leq D$ in this case.

Now by Corollary \ref{C3}, we need to determine $\mathcal{D}_m$
for $m=D=h$ and $m=D-1=h-1$.

To this end, we cite the following main result of \cite{Gao}.

\begin{prop}\label{ZS} Let $S=(a_1,\cdots,a_{n-k})$ be a sequence
of $n-k$ (not necessarily distinct) elements in $\Z_n=\Z/n\Z$.
Suppose $1\leq k\leq n/6+1$ and that $0$ cannot be expressed as a
sum over a nonempty subsequence of $S$; then there exist an
integer $c$ coprime to $n$ and a permutation $\sigma$ of the set
$\{1,2,\cdots,n-k\}$ such that $ca_{\sigma (i)}=1$ for
$i=1,\cdots,n-2k+1,$ and
$\sum_{i=n-2k+2}^{n-k}|a_{\sigma(i)}|_n\leq 2k-2,$ where $|x|_n$
denotes the least positive inverse image of $x$ under the natural
homomorphism from the additive group of integers onto $\Z_n.$

In particular, there are at least $n-2k+1$ terms in $S$ which are
relatively prime to $n$ and all congruent to one another modulo
$n$.
\end{prop}

We use this result to prove the following lemma.

\begin{lem}\label{L4}  Suppose $\Cl=\la \fc\ra$. Then

$$\mathcal{D}_D=\{\kappa_k : 1\leq k\leq h, (k,h)=1\}, $$ where
$\kappa_k : \Cl\longrightarrow\N_0$ with $\kappa_k(\fc^k)=h$ and
$\kappa_k(\fc^\ell)=0$ otherwise;

$$\mathcal{D}_{D-1}=\{\lambda_k : 1\leq k\leq h, (k,h)=1\}, $$ where
$\lambda_k : \Cl\longrightarrow\N_0$ with $\lambda_k(\fc^k)=h-2,$
$\lambda_k(\fc^{2k})=1$, and  $\lambda_k(\fc^\ell)=0$ otherwise.

\end{lem}

\begin{proof}  We start by determining the elements of
$\mathcal{D}_D$. Suppose $\fc_1,\cdots,\fc_h\in\Cl$ and
$\prod_{i=1}^h\fc_i\overset{min}{=}1.$ Then the $h$ sequences
$S_j=(\fc_1,\cdots,\hat{\fc_j},\cdots,\fc_h)$ (where $\fc_j$ is
omitted) satisfy the hypotheses of Proposition \ref{ZS} with
$k=1$. Hence in each $S_j$ there are at least $h-1$ terms which
are equal and generating $\Cl$.  Hence, we must have $\fc_1=
\cdots = \fc_h= \fc$ and $\la \fc\ra=\Cl.$ Hence $\mathcal{D}_D$
is as stated above.

Now consider $\mathcal{D}_{D-1}$. Suppose
$\fc_1,\cdots,\fc_{h-1}\in\Cl$ and
$\prod_{i=1}^{h-1}\fc_i\overset{min}{=}1.$ Then the $h-1$
sequences $S_j=(\fc_1,\cdots,\hat{\fc_j},\cdots,\fc_{h-1})$
satisfy the hypotheses above with $k=2$ provided $h\geq 6$. (For
$h<6$ the lemma follows by a straightforward calculation.)  Hence,
assume $h\geq 6$ in which case in each $S_j$ there are at least
$h-3$ terms which are equal and generate $\Cl$. But then, without
loss of generality, $\fc_1=\cdots =\fc_{h-2}=\fc$ where $\la
\fc\ra=\Cl.$ Thus $\fc^{h-2}\fd=1$ for some $\fd\in\Cl$; whence
$\fd=\fc^2$, as desired.

\end{proof}

This lemma along with Corollary \ref{C3} and Theorem \ref{T1}
yields the following proposition.

\begin{prop}\label{P6} Let $K$ be an algebraic number field with
cyclic class group $\Cl=\la \fc\ra$ of order $h>1$. Then
$$M(x)=\frac{x}{\log x}\left(C(\log\log x)^{h-1}+B(\log\log
x)^{h-2}\right)+O\left(E(x)\right),$$ where

$$C=\frac{\varphi(h)}{(h-1)!h^h},$$ and
$$B=\frac{\varphi(h)}{(h-2)!h^h}\gamma+\frac{h-1}{h^{h-1}}
\left(\frac{\varphi(h)}{(h-2)!}a(h)+\frac1{(h-1)!}\sum_{k=1 \atop
(k,h)=1}^h g_{\fc^k}(1)\right),$$ where $a(h)=1/2$, if $h=3$, and
$a(h)=1$, otherwise;  and where $g_{\fc}$ is as appears in
\rm{Corollary \ref{C5.1}}.
\end{prop}

The proof follows immediately from Corollary \ref{C3} and Theorem
\ref{T1}. (Notice that when $h=3$, $|\mathcal{D}_2|=1$, not
$\varphi(h)$.)

 We now give an (partially) arithmetic
interpretation of $$\sum_{k=1 \atop (k,h)=1}^h g_{\fc^k}(1).$$
First, we introduce some notation.

Once again assume $K$ has cyclic class group $\Cl=\la \fc\ra$ and
let $L$ be the Hilbert class field of $K$. For each divisor $d$ of
$h$ let $L_d$ denote the intermediate field in the extension $L/K$
of degree $d$ over $K$. (Since by class field theory
$\Gal(L/K)\simeq \Cl$ and $\Cl$ is cyclic, $L_d$ is uniquely
determined.) Notice in particular that $L_1=K$ and $L_h=L$.
Finally, let $a_{L_d}$ be the residue of the Dedekind zeta
function $\zeta_{L_d}(s)$ at $s=1$.

\begin{thm}\label{T2} Given the assumptions of the previous
paragraph,  $$\sum_{k=1 \atop (k,h)=1}^h
g_{\fc^k}(1)=\sum_{d|h}\frac{\mu(d)}{d}\log a_{L_d}-\underset{\la
[\fp^m]\ra=\Cl}{\sum_{m\geq 2}\sum_{\fp}}\frac1{mN\fp^m}.$$
\end{thm}
\begin{proof} By Corollary \ref{C5.1} we have
 $$\sum_{k=1 \atop (k,h)=1}^h g_{\fc^k}(s)=\frac{\varphi(h)}{h}\log
((s-1)\zeta_K(s))+\frac1{h}\beta(s)-\sum_{m=2}^\infty \sum_{k=1
\atop (k,h)=1}^h\sum_{\fp\atop \fp^m\in
\fc^k}\frac1{mN\fp^{ms}},$$ where $$\beta(s)=\sum_{\chi\atop
\chi\not=1}\sum_{k=1 \atop (k,h)=1}^h\overline{\chi}(\fc^k) \log
L(s,\chi).$$

For $j=0,\cdots,h-1$ let $\chi_j$ be the character on $\Cl$
determined by $\chi_j(\fc)=\zeta_h^j$ for $\zeta_h$ a primitive
$h$th root of unity. More generally, let $\chi_{d,j}$ be the
character on $\Cl$ determined by $\chi_{d,j}(\fc)=\zeta_d^j,$ for
any positive integer $d$ dividing $h$. Also let $$c_n(j)=\sum_{k=1
\atop (k,h)=1}^h\zeta_h^{jk},$$ the usual Ramanujan sum. Then
$$\beta(s)=\sum_{j=1}^{h-1}c_h(-j)\log L(s,\chi_j).$$ But the
Ramanujan sum has the explicit representation (see, e.g.
\cite{GR}, p. 238)
$$c_n(j)=\varphi(h)\frac{\mu(h/(h,j))}{\varphi(h/(h,j))}, $$ and
thus $$\beta(s)=\varphi(h)\sum_{\nu|h}
\frac{\mu(\nu)}{\varphi(\nu)}\sum_{j=1\atop (h,j)=h/\nu}^{h-1}\log
L(s,\chi_j)$$ $$=\varphi(h)\sum_{\nu|h}
\frac{\mu(\nu)}{\varphi(\nu)}\sum_{j=1\atop (h,j)=h/\nu}^{h}\log
L(s,\chi_j)-\varphi(h)\log \zeta_K(s).$$ Now, by \cite{Lg}, p.
230, we have $$\log \zeta_{L_d}(s)=\sum_{\nu|d}\sum_{j\bmod v\atop
(j,\nu)=1}\log L(s,\chi_{\nu,j}).$$ But then by M\"obius
inversion, $$\sum_{j\bmod h\atop (h,j)=h/\nu}\log
L(s,\chi_j)=\sum_{j\bmod v\atop (j,\nu)=1}\log
L(s,\chi_{\nu,j})=\sum_{d|\nu}\mu(\nu/d) \log \zeta_{L_d}(s).$$
Thus $$\varphi(h)\sum_{\nu|h}
\frac{\mu(\nu)}{\varphi(\nu)}\sum_{j=1\atop (h,j)=h/\nu}^{h}\log
L(s,\chi_j)=\varphi(h)\sum_{\nu|h}\sum_{d|\nu}\frac{\mu(\nu)}{\varphi(\nu)}\mu(\frac{\nu}{d})\log
\zeta_{L_d}(s)$$
$$=\varphi(h)\sum_{d|h}\log\zeta_{L_d}(s)\sum_{\nu|h\atop
d|\nu}\frac{\mu(\nu)\mu(\nu/d)}{\varphi(\nu)}=
\varphi(h)\sum_{d|h}\log\zeta_{L_d}(s)\mu(d)\sum_{\nu|h\atop
d|\nu}\frac{\mu^2(\nu)}{\varphi(\nu)}$$
$$=\varphi(h)\sum_{d|h}\log\zeta_{L_d}(s)\frac{h}{\varphi(h)}\frac{\mu(d)}{d}=
h\sum_{d|h}\frac{\mu(d)}{d}\log\zeta_{L_d}(s),$$ since
$$\sum_{\nu|h\atop
d|\nu}\frac{\mu^2(\nu)}{\varphi(\nu)}=\frac{h}{\varphi(\nu)}\frac{\mu^2(d)}
{d},$$ see, for example \cite{BOS}, Lemma 3.  Hence
$$\beta(s)=h\sum_{d|h}\frac{\mu(d)}{d}\log\zeta_{L_d}(s)-\varphi(h)\log\zeta_K(s).$$

Now notice that $$\lim_{\sigma\rightarrow
1^+}\beta(s)=~~~~~~~~~~~~~~~~~~~~~~~~~~~~~~~~~~~~~~~~~~~~$$ $$
h\sum_{d|h}\frac{\mu(d)}{d}\log (s-1)\zeta_{L_d}(s)-\varphi(h)\log
(s-1)\zeta_K(s)-\left(h\sum_{d|h}\frac{\mu(d)}{d}-\varphi(h)\right)\log
(s-1)=$$ $$h\sum_{d|h}\frac{\mu(d)}{d}\log a_{L_d} -
\varphi(h)\log a_K,$$ since $\varphi(h)=h\sum_{d|h}\mu(d)/d.$

This gives us the result.
\end{proof}

\section{Examples}

The coefficient $C$ of $M(x)$ depends on the class group of $K$,
more precisely, on the Davenport constant and the order of the
class group. On the other hand, the coefficient $B$ seems to
depend more intrinsically on the arithmetic for the field $K$. In
this section we consider approximating $B$ for two imaginary
quadratic number fields of class number 2, namely,
$K_1=\Q(\sqrt{-5}\,)$ and $K_2=\Q(\sqrt{-15}\,)$  to see if the
$B$ are unequal. But before we carry out the calculations in these
special cases, we consider Proposition \ref{P6} for the case where
$h=2$.

\begin{cor}\label{C8} Let $K$ be a number field with class number $2$. Denote
by $\fc$ the nonprincipal ideal class of  $\Cl$. Finally, let $L$
be the Hilbert class field of $K$. Then
$$M(x)=\frac1{4}\frac{x}{\log x}\log\log x
+\frac1{4}(2(1+g_\fc(1))+\gamma)\frac{x}{\log
x}+O\left(\frac{x}{(\log x)^{3/2}}\right),$$ where $\gamma$ is
Euler's constant and
$$g_\fc(1)=\log a_K-\frac1{2}\log a_L-\sum_{m\geq
3\atop m \equiv 1 (2)}\sum_{\fp\in\fc}\frac1{mN\fp^m}.$$
\end{cor}

We need to compute $a_K$, $a_L$, and $S=\sum_{m\geq 3\atop m
\equiv 1 (2)}\sum_{\fp\in\fc}\frac1{mN\fp^m}$.

To this end, let $F$ be any algebraic number field. Then the
residue of $\zeta_F(s)$ at $s=1$ is
$$a_F=\frac{2^{r_1}(2\pi)^{r_2}R_Fh_F}{w_F\sqrt{|d_F|}},$$
where $r_1$ and $r_2$ are the number of inequivalent real and
complex embeddings of $F$ into $\C$, respectively; $R_F$ is the
regulator of $F$; $h_F$ its class number; $w_F$ the number of
roots of unity in $\OO_F$; and $d_F$ is the discriminant of $F$.

For $K_1=\Q(\sqrt{-5}\,)$, $r_1=0$, $r_2=1$, $R_{K_1}=1$,
$w_{K_1}=2$, and $d_{K_1}=-20$, and hence
$$a_{K_1}=\frac{\pi}{\sqrt{5}}.$$

For $K_2=\Q(\sqrt{-15}\,)$, $r_1=0$, $r_2=1$, $R_{K_2}=1$,
$w_{K_2}=2$, and $d_{K_2}=-15$, and hence
$$a_{K_2}=\frac{2\pi}{\sqrt{15}}.$$

The Hilbert class fields of $\Q(\sqrt{-5}\,)$ and
$\Q(\sqrt{-15}\,)$ are $L_1=\Q(\sqrt{-5},\sqrt{5}\,)$ and
$L_2=\Q(\sqrt{-15},\sqrt{5}\,)$, respectively. To compute
$a_{L_i}$ in these two cases, we first notice that $r_1=0$ and
$r_2=2$. To compute the other invariants, we shall use the fact
that $L_i$ are CM-fields, which will allow us to compute the
regulators $R_L$, and the fact that $\Gal(L_i/\Q)\simeq C(2)\times
C(2)$, the Klein four group, which will give us a way to compute
the class numbers.

To this end, let $L^+=L\cap \R=\Q(\sqrt{5}\,)$ in both cases
$L=L_i$. Now $R_{L^+}=\log ((1+\sqrt{5})/2)$ and by Proposition
4.16 of \cite{Wash} (for example)
$R_L=(1/Q)2\log((1+\sqrt{5})/2)$, where $Q=(E_L:W_LE_{L^+})\in
\{1,2\}$ with $E_F$ and $W_F$ the group of units, respectively,
roots of unity in $\OO_F$ for any number field $F$. But in our two
cases, $Q=1$; see \cite{Lem2} Theorem 1. Thus in both cases
$$R_L=2\log\frac{1+\sqrt{5}}{2}.$$

By Proposition 17 on page 68 in \cite{Lg} (for example) we see
$$ d_{L_1}=20^2$$ and $$d_{L_2}=15^2.$$

Finally, to compute the class numbers, we use Kuroda's class
number formula: $$h_L=\frac1{2}q(L)h_1h_2h_3,$$ where the $h_i$
are the class numbers of the three quadratic subfields of $L$, and
$q(L)=(E_L:E_1E_2E_3)$ with $E_i$ the group of units in the
quadratic subfields, cf. for example \cite{Lem1}. In our cases,
$h_1h_2h_3=2$ and since $L/K$ is unramified $q(L)=1$, \cite{Lem2},
Theorem 1. Hence in both cases $$h_L=1.$$

Therefore,
$$a_{L_1}=\frac{\pi^2}{10}\log\left(\frac{1+\sqrt{5}}{2}\right),$$ and
$$a_{L_2}=\frac{4\pi^2}{15}\log\left(\frac{1+\sqrt{5}}{2}\right).$$

Next, we need to approximate the two series $$S_i:=\sum_{m\geq
3\atop m \equiv 1 (2)}\sum_{\fp\in\fc_i}\frac1{mN\fp^m}$$ for the
fields $K_i$, $i=1,2$ and where $\Cl(K_i)=\la \fc_i\ra.$ Now,
since
$$\sum_{m\geq
3\atop m \equiv 1
(2)}\frac{1}{mz^m}=\frac1{2}\left(\frac{\log(z+1)}{\log(z-1)}-\frac2{z}\right),$$
we see that
$$S=\sum_{\fp\in\fc}\sum_{m\geq
3\atop m \equiv 1 (2)}\frac1{mN\fp^m}=\sum_{\fp\in\fc}\frac1{2}
\left[\log\left(\frac{N\fp+1}{N\fp-1}\right)-\frac{2}{N\fp}\right].$$

We now truncate the series $S$ at $N\fp<x$ for $x>3$ and estimate
the truncation error by a little elementary calculus. To this end,
we write $$S=S(x)+E(x),$$ where $$S(x):=\sum_{\fp\in\fc\atop
N\fp<x}\frac1{2}
\left[\log\left(\frac{N\fp+1}{N\fp-1}\right)-\frac{2}{N\fp}\right]$$
and $$E(x)=\sum_{\fp\in\fc\atop N\fp\geq x}\sum_{m\geq 3\atop m
\equiv 1 (2)}\frac1{mN\fp^m}.$$ Now, notice that
$$\sum_{\fp\in\fc\atop N\fp\geq x}\frac1{mN\fp^m}<\sum_{k\geq
x}\frac2{mk^m}<\int_{x-1}^\infty\frac2{mt^m}dt=\frac2{m(m-1)(x-1)^{m-1}},$$
since $N\fp=k$ can occur at most twice (when $p\OO_K$ splits where
$\fp | p$). Hence $$|E(x)|\leq
\sum_{m=3}^\infty\frac2{m(m-1)(x-1)^{m-1}}$$ $$
<\frac1{3}\sum_{m=3}^\infty\frac1{(x-1)^{m-1}}=
\frac1{3(x-1)(x-2)}<\frac1{3(x-2)^2}.$$

Next, to approximate $S(x)$,  we need to find out which prime
ideals are not principal in $\OO_{K_i}$.  But since the $L$ are
abelian over $\Q$, the prime ideals that are nonprincipal are
determined by congruences on the rational primes contained in
these ideals. We now review this procedure. We consider the case
$K=K_1$. Let $(d_K/\;)$ denote the Kronecker symbol and suppose
$\fp|p$, $p$ a positive rational prime; then $(d_K/p)=-1$ if and
only if $\fp=p\OO_K$, i.e. $p$ is inert in $K$.  By reciprocity,
this occurs when $p\equiv 11,13,17,19 \bmod 20$. Hence in this
case, $\fp$ is a principal ideal. Therefore, if $\fp$ is
nonprincipal, then $(d_K/p)=1$ or $0$, i.e. $p$ splits or is
ramified, respectively, in $K$. Suppose first that
$p\OO_K=\fp\overline{\fp},$ for distinct prime ideals $\fp$ and
$\overline{\fp}$. Then by properties of the Hilbert class field of
$K$, $\fp$ and $\overline{\fp}$ are nonprincipal if and only if
$\fp\OO_L$ is a prime ideal. For $K_1$, this happens if and only
if  $(-20/p)=1$ and $(-1/p)=-1$, i.e. if and only if $p\equiv
3,7\bmod 20.$ (Notice then that $\fp$ and $\overline{\fp}$ are
principal when $p\equiv 1,9 \bmod 20.$) On the other hand, the
ramified primes in $K_1$ are the (unique) prime ideals dividing 2
and 5. But if $\fp|5$ then $\fp=\sqrt{-5}\OO_{K_1}$, which is
principal; whereas if $\fp|2$, then $\fp$ is nonprincipal, since
otherwise $\fp=(a+b\sqrt{-5})\OO_{K_1}$ for some $a,b\in\Z$, in
which case $2=N\fp=a^2+5b^2$, which is absurd. Similarly, for
$K_2$, $\fp$ is nonprincipal when $(-15/p)=1$ and $(-3/p)=-1$,
i.e. when $\fp|p$ where $p\equiv 2,8\bmod 15$, and for  $p=3,5$
(ramified case). (On the other hand, $\fp$ is principal whenever
$p\equiv 1,4,7,11,13,14 \bmod 15$.)

Thus,
$$S_1(x)=\frac1{2}
\left[\log 3-1\right]+\sum_{p<x\atop p\equiv 3,7
(20)}\left[\log\left(\frac{p+1}{p-1}\right)-\frac{2}{p}\right]$$
and
$$S_2(x)=\frac1{2}\log 3-\frac1{3}-\frac1{5}+
\sum_{p<x\atop p\equiv 2,8
(15)}\left[\log\left(\frac{p+1}{p-1}\right)-\frac{2}{p}\right].$$

To approximate $S$ to four decimal places, say, we use
$|E(x)|<1/(3(x-2)^2)<.5\times 10^{-4}$, in which case we may take
$x=84$. Then notice that $p\equiv 3,7\bmod 20$ with $p<84$ if and
only if $p=3,7,23,43,47,67,83.$ Also $p\equiv 2,8\bmod 15$ with
$p<84$ if and only if $p=2,17,23,47,53,83.$ Hence $S_1\approx
S_1(84)\approx 0.077827$ and $S_2\approx S_2(84)\approx 0.232435$
good to four decimal places.

On the other hand,  $$\log a_{K_i}-\frac1{2}\log a_{L_i}\approx
0.71229745 \mbox{~~~and~~~} 0.36572386 $$ for $i=1,2$,
respectively.

Therefore $$g_{\fc_1}(1)\approx 0.6343$$ and
$$g_{\fc_2}(1)\approx 0.1333.$$

This shows that the coefficient $B$ differs for these two
quadratic number fields.

Finally, as promised in the introduction, we characterize the
primes and irreducibles in $\Z[\sqrt{-5}\,]$ and
$\Z[\sqrt{-15}\,]$ in terms of  rational primes.

\begin{prop} a) An element $\pi$ is prime in $\Z[\sqrt{-5}\,]$ if and
only if $\pi |p$ a positive rational prime such that $p=5$ or
$p\equiv 1,9,11,13,17,19 \bmod 20$;

b) $\pi$ is prime in $\Z[\sqrt{-15}\,]$ if and only if  $p\equiv
1,4,7,11,13,14 \bmod 15$;

c) $\alpha$ is irreducible but not prime in $\Z[\sqrt{-5}\,]$ if
and only if $|N(\alpha)|=p_1p_2$ where $p_1,p_2$ are positive
rational primes such that $p_i=2$ or $p_i\equiv 3,7 \bmod 20$;

d) $\alpha$ is irreducible but not prime in $\Z[\sqrt{-15}\,]$ if
and only if $|N(\alpha)|=p_1p_2$ where $p_i=3,5$ or $p_i\equiv 2,8
\bmod 15$.

\end{prop}

\vspace{.2in}

{\bf Address of the authors:}

\medskip

Bradley, \"Ozl\"uk, Snyder

Department of Mathematics and Statistics

University of Maine

Orono, Maine 04469

~~~~~~~~~~~~~~and

Research Institute of Mathematics

 Orono, ME 04473 \vspace{.15in}

 \smallskip

Rozario

19 Balsam Drive

Bangor, Maine  04401

\bigskip

 {\bf e-mail addresses}

 \medskip

 bradley@math.umaine.edu

 ozluk@math.umaine.edu

 rozario@umich.edu

 snyder@math.umaine.edu

\end{document}